\documentclass{amsart}

\usepackage{amsmath,amscd,amssymb,amsfonts,amsthm,enumerate}
\usepackage[colorlinks=true,linkcolor=blue,citecolor=blue]{hyperref}
\usepackage{cleveref}
\usepackage{color}
\usepackage{graphicx}

\usepackage[numbers,sort&compress]{natbib}

\DeclareMathOperator{\Ass}{Ass}
\DeclareMathOperator{\chara}{char}
\DeclareMathOperator{\depth}{depth}
\DeclareMathOperator{\Min}{Min}

\DeclareMathOperator{\Spec}{Spec}
\DeclareMathOperator{\reg}{reg}
\DeclareMathOperator{\Tor}{Tor}
\DeclareMathOperator{\um}{um}

\newcommand{\Gcc}{\mathcal{G}}
\newcommand{\mm}{\mathfrak{m}}
\newcommand{\nn}{\mathfrak{n}}
\newcommand{\pp}{\mathfrak{p}}

\theoremstyle{plain}
\newtheorem{thm}{Theorem}[section]

\newtheorem{prop}[thm]{Proposition}
\newtheorem{lem}[thm]{Lemma}
\newtheorem{cor}[thm]{Corollary}

\newtheorem{quest}[thm]{Question}

\theoremstyle{definition}

\newtheorem{notn}[thm]{Notation}
\newtheorem{ex}[thm]{Example}
\newtheorem{rem}[thm]{Remark}

\numberwithin{equation}{section}

\newcommand{\kk}{\Bbbk}
\newcommand{\up}[1]{{}^{\tt #1}\!}

\newcommand{\Z}{\mathbb{Z}}
\newcommand{\Rcc}{\mathcal{R}}

\begin{document}


\title[Ordinary and symbolic powers of fiber products]{On the ordinary and symbolic powers\\ of fiber products}

\author[H.V. Do]{Hoang Viet Do}
\address{Institute of Mathematics \\ Vietnam Academy of Science and Technology, 18 Hoang Quoc Viet \\ Hanoi, Vietnam}
\email{vietdohoang10@gmail.com}

\author[H.D. Nguyen]{Hop D. Nguyen}
\address{Institute of Mathematics \\ Vietnam Academy of Science and Technology, 18 Hoang Quoc Viet \\ Hanoi, Vietnam}
\email{ndhop@math.ac.vn}
\email{ngdhop@gmail.com}

\author[S.A. Seyed Fakhari]{Seyed Amin Seyed Fakhari}
\address{Departamento de Matem\'aticas, Universidad de los Andes, Bogot\'a, Colombia}
\email{s.seyedfakhari@uniandes.edu.co}


\begin{abstract}
We completely determine the depth and regularity of symbolic powers of the fiber product of two homogeneous ideals in disjoint sets of variables, given knowledge of the symbolic powers of each factor.  Generalizing previous joint work with Vu, we provide exact, characteristic-independent formulas for the depth and regularity of ordinary powers of such fiber products.
\end{abstract}






\thanks{}


\maketitle


\section{Introduction}
Let $A, B$ be \emph{standard graded} algebras over a field $\kk$, i.e. each of them is a finitely generated graded $\kk$-algebra generated by elements of degree 1. Given the presentations of $A$ and $B$ as quotients of polynomial rings $A=R/I, B=S/J$, we obtain the corresponding presentations of the tensor product $A\otimes_\kk B$ and the fiber product $A\times_\kk B$. Let $T=R\otimes_\kk S$, $\mm, \nn$ be the graded maximal ideals of $R,S$, respectively. Then $A\otimes_\kk B \cong T/(I+J)$, where we identify ideals of $R$ and $S$ with their extensions to $T$. Similarly, $A\times_\kk B\cong T/(I+J+\mm \nn)$. Hence the study of tensor products and fiber products of standard graded $\kk$-algebras is equivalent to the study of the \emph{sums} $I+J$ and \emph{fiber products} $I+J+\mm\nn$ of homogeneous ideals in disjoint sets of variables. The research on ordinary and symbolic powers of the sum $I+J$, assuming the knowledge of the corresponding powers of the summands $I$ and $J$, were taken up by many researchers; see, e.g., \cite{BHJT21, HJKN, HTT16, HNTT, NgHo22}. Invariants of fiber products of $\kk$-algebras and their defining ideals have attracted attention of many researchers; see, for example, \cite{AAF12,CR10, DK75, FJP23, G21, Les81,Mo09, NSW17, NgV19b}.

The depth, regularity, and symbolic analytic spread, of symbolic powers of various class of ideals have been considered by many authors; a partial list of fairly recent work is \cite{DM21, DHNT, HJKN, HKTT, HNTT, JK20, MNPTV21, MT19, MTV24, MV24, NgT19, OR19, SF18, SF19, SF20, SF24}. In this paper, we investigate the problem of determining  the depth and regularity of symbolic powers of a fiber product of ideals, given knowledge of the individual factors. Our first main result is the following statement, where $\kk$ is a field of arbitrary characteristic.
\begin{thm}[= \Cref{thm_depth_reg}]
\label{thm_main}
Let $R$ and $S$ be positive dimensional standard graded polynomial rings over $\kk$, with graded maximal ideals $\mm$ and $\nn$. Let $I\subseteq \mm^2, J\subseteq \nn^2$ be homogeneous ideals, such that $\min\{\depth(R/I),$ $\depth(S/J)\}\ge 1$. Let $T=R\otimes_\kk S$ and $F=I+J+\mm \nn$. Then for each integer $s\ge 1$, we have equalities
\begin{align*}
\depth(T/F^{(s)}) &=1,\\
\reg(F^{(s)})   &=\max\limits_{i\in [1,s]} \left\{2s,\reg(I^{(i)})+s-i, \reg(J^{(i)})+s-i\right \}.
\end{align*}
\end{thm}
The proof relies upon the following decomposition formulas for symbolic powers of $F$, which might be of independent interest: With the notations and hypotheses of \Cref{thm_main}, for all $s\ge 1$, we have (see \Cref{lem_decomposition} and \Cref{lem_decomposition_as_sum}):
\begin{align*}
F^{(s)}&=(I+\nn)^{(s)} \cap (J+\mm)^{(s)}\\
       &=\sum_{i=0}^s\sum_{t=0}^s (I^{(i)}\cap \mm^{s-t})(J^{(t)}\cap \nn^{s-i}). 
\end{align*}
Recall that an ideal is \emph{unmixed} if it has no embedded associated primes.
\begin{cor}
\label{cor_main_regunmixedcase}
Keep using the hypotheses of \Cref{thm_main}. Assume moreover that $I$ is a non-zero unmixed monomial ideal of $R$. Then for all $s\ge 1$, there is an equality
$$
\reg (F^{(s)})=\max \limits_{i\in [1,s]} \left\{\reg(I^{(i)})+s-i, \reg(J^{(i)})+s-i \right\}.
$$
\end{cor}
In both of the above results, we are concerned with the case $\min\{\depth(R/I),$ $\depth(S/J)\}\ge 1$. How about the case $\min\{\depth(R/I), \depth(S/J)\}=0$? It is not hard to show that in this case, $\depth(T/F)=0$, and therefore the ordinary and symbolic powers of $F$ coincide. Thus in order to study the depth and regularity of symbolic powers of fiber product in this case, we are led to the consideration of those invariants for the corresponding ordinary powers. This problem was considered by the second named author and Vu in \cite{NgV19b}. But the main results in \emph{ibid.} are characteristic-dependent, and they work mainly in characteristic zero. In this paper, we introduce an entirely different approach to solve the problem completely in all characteristics. Thus for the depth of ordinary powers of fiber products, we have
\begin{prop}[= \Cref{prop_depthordinarypower_0}]
\label{prop_depthordinarypower_main}
Let $\kk$ be a field of arbitrary characteristic. Let $I\subseteq \mm^2, J\subseteq \nn^2$ be homogeneous ideals of $R,S$, resp., at least one of which is non-zero. Then for all $s\ge 2$, there is an equality $\depth(T/F^s)=0$.
\end{prop}

For regularity of ordinary powers of fiber products, we have

\begin{thm}[= \Cref{thm_reg_ordinarypow}]
\label{thm_regordinarypow_main}
Let $\kk$ be a field of arbitrary characteristic. Let $I\subseteq \mm^2$ and $J\subseteq \nn^2$ be homogeneous ideals of $R$ and $S$, resp. Then for  all $s\ge 1$, there is an equality
\begin{align*}
\reg F^s &= \max \limits_{i\in [1,s]} \left\{2s,\reg(\mm^{s-i}I^i)+s-i, \reg(\nn^{s-i}J^i)+s-i\right\}.
\end{align*}
If moreover either $I$ or $J$ is non-zero, then for all $s\ge 1$, there is an equality
\[
\reg F^s = \max \limits_{i\in [1,s]} \left\{\reg(\mm^{s-i}I^i)+s-i, \reg(\nn^{s-i}J^i)+s-i\right\}.
\]
\end{thm}
Special cases of the last two results were obtained in \cite[Theorems 5.1 and 6.1]{NgV19b}, where the hypothesis that either $\chara \kk=0$ or $I$ and $J$ are both monomial ideals, is required. To prove the formula of \Cref{thm_reg_ordinarypow}, the approach of \cite{NgV19b} is via \emph{Betti splitting}, based on the fact that, assuming either $\chara \kk=0$ or $I\subseteq \mm^2$ is a monomial ideal, the map $\mm^{s-t}I^t\to \mm^{s-t+1}I^{t-1}$ is Tor-vanishing for every $1\le t\le s$. Recall that a map $M\to N$ between $R$-modules is \emph{Tor-vanishing} if the induced map on Tor modules $\Tor^R_i(M,\kk)\to \Tor^R_i(N,\kk)$ is zero for all $i$. We note that this approach is not applicable if $\chara \kk$ is positive (see \Cref{rem_noTor-vanishing} for more details).

The proof of \Cref{thm_regordinarypow_main} exploits the special structure of the zero-th local cohomology $H^0_{\mm T+\nn T}(T/F^s)$, which is non-trivial for all $s\ge 2$ by \Cref{prop_depthordinarypower_main}. Our main observation in the proof of this theorem is that $T/F^s$ contains the finite length submodule $\left((I+\nn)^s\cap (J+\mm)^s\right)/F^s$, whose regularity is strongly related to that of powers of $I$ and $J$. The analysis of the regularity of $\left((I+\nn)^s\cap (J+\mm)^s\right)/F^s$ is the main novelty of our approach, which helps us to avoid using characteristic-dependent arguments.

Finally, we state our result on the depth and regularity for symbolic powers of $F$ in the case $\min \left\{\depth(R/I),\depth(S/J)\right\}=0$, complementing \Cref{thm_main}. Again we will assume that the hypotheses of \Cref{thm_main} are in force, and in particular, $\dim R, \dim S \ge 1$.
\begin{thm}[= \Cref{thm_depth_zerodepthcase}]
\label{thm_zerodepthcase_main}
Assume that $\min \left\{\depth(R/I),\depth(S/J)\right\}=0$. Then for all $s\ge 1$,  there are equalities $F^{(s)}=F^s$, and
\begin{align*}
\depth(T/F^{(s)})&=0,\\
\reg F^{(s)} &= \max \limits_{i\in [1,s]} \left\{\reg(\mm^{s-i}I^i)+s-i, \reg(\nn^{s-i}J^i)+s-i\right\}.
\end{align*}
\end{thm}
{\bf Organization}. The decompositions of symbolic powers of the fiber product of $I$ and $J$, as intersection and as sum, in the case $\min\left\{\depth(R/I),\depth(S/J)\right \}\ge 1$, are established in \Cref{sect_decomposition}. In \Cref{sect_symbpow_positivedepth}, we use the decomposition formulas of the last section to determine the depth and regularity of $F^{(s)}$ when the depths of $R/I$ and $S/J$ are both positive. \Cref{sect_ordinarypowers} is dedicated to the proofs of \Cref{prop_depthordinarypower_main} and \Cref{thm_regordinarypow_main}, from which we deduce \Cref{thm_zerodepthcase_main}. In the last \Cref{sect_remarks}, we discuss how our results can be adapted to the case of ``minimal'' symbolic powers, and propose some questions related to the main results.


\section{Background}
For standard notions and results of commutative algebra, we refer to \cite{BH98, Eis95, Pe11}. 

Let $\kk$ be a field, $R$ and $S$ be standard graded algebras over $\kk$. We set $T=R\otimes_{\kk}S$.  The following lemma is folklore; see, e.g., \cite[Lemma 3.1]{HNTT}. By abuse of notations, we use the same symbols to denote ideals of $R$ and $S$ and their extensions to $T$.
\begin{lem}
\label{lem_intersection}
In $T$, there is an identity $I\cap J=IJ$.
\end{lem}

\begin{lem}[{\cite[Theorem 2.5 and its proof]{HNTT}}]
\label{lem_Ass_tensor}
Let $M, N$ be non-zero finitely generated modules over $R,S$, resp. For $P\in \Spec(T)$, let $\pp_1=P\cap R$, $\pp_2=P\cap S$.
\begin{enumerate}[\quad \rm(i)]
\item $P\in \Ass_T(M\otimes_\kk N)$ if and only if the following conditions hold:
\[
\pp_1 \in \Ass_R(M), \pp_2 \in \Ass_S(N), P \in \Min_T(T/(\pp_1+\pp_2)).
\]
\item $P\in \Min_T(M\otimes_\kk N)$ if and only if the following conditions hold:
\[
\pp_1 \in \Min_R(M), \pp_2 \in \Min_S(N), P \in \Min_T(T/(\pp_1+\pp_2)).
\]
\end{enumerate} 
\end{lem}

Let $R$ be a standard graded $\kk$-algebra, and $M$ a finitely generated graded $R$-module. Then \emph{relative Castelnuovo--Mumford regularity} of $M$ over $R$ is
\[
\reg_R M:=\sup\{j-i\mid \Tor^R_i(M,\kk)_j\neq 0\}.
\]
The  \emph{absolute Castelnuovo--Mumford regularity} of $M$ is defined in terms of local cohomology supported at $\mm$ as
\[
\reg M:= \sup\{i+j\mid H^i_\mm(M)_j\neq 0\}.
\]
When $R$ is a regular ring, then thanks to local duality, both notions of regularity coincide: $\reg_R M= \reg M$. The next result is folklore; see, e.g., \cite[Lemma 2.3]{NgV19a}.
\begin{lem}
\label{lem_tensor}
Let $R, S$ be standard graded $\kk$-algebras, and $M, N$ be finitely generated graded modules over $R, S$, respectively. Then for $T=R \otimes_\kk S$, there is an equality $\reg_T (M\otimes_\kk N)=\reg_R M+ \reg_S N$.
\end{lem}

The following is essentially \cite[Lemma 5.3]{NgV19b}, except that the crucial hypothesis ``$R$ is a polynomial ring'' that was mistakenly omitted in part (ii) of \cite[Lemma 5.3]{NgV19b} is added below. In the result, for a finitely generated graded $R$-module $M$, $d(M)$ denotes the maximal degree of a minimal homogeneous generator of $M$.
 
\begin{lem}[Eisenbud--Ulrich]
\label{lem_regmM}
Let $(R,\mm)$ be a standard graded $\kk$-algebra.  Let $M\neq 0$ be a finitely generated graded $R$-module such that $\depth M\ge 1$.
\begin{enumerate}[\quad \rm (i)]
 \item For all $s\ge 1$, there is an equality $\reg (\mm^sM)= \max \left\{\reg M, \reg \dfrac{M}{\mm^s M}+1\right\}$.
\item \textup{(See \c{S}ega \cite[Theorem 3.2]{Se01})} Assume furthermore that $R$ is a standard graded polynomial ring over $\kk$, and $M$ is generated in a single degree. Then for all $s\ge 1$, there is an equality
\[
\reg (\mm^sM)= \max \left\{\reg M, s+d(M)\right\}.
\]
In particular, for all $s\ge \reg M-d(M)$, $\mm^sM$ has a linear resolution.
\end{enumerate}
\end{lem}


\subsection{Symbolic powers}
For a recent survey on symbolic powers of ideals, we refer to \cite{DDSG+18}.

In this paper, we use symbolic powers that are defined in terms of associated primes (as in \cite{DDSG+18}), not minimal primes. The former notion is more general than the latter, as we will explain below. Let $R$ be a noetherian ring, and $I$ an ideal of $R$. For an integer $s\ge 1$, define the $s$-th \emph{symbolic power} of $I$ to be
\[
I^{(s)}=\bigcap_{P \in \Ass(R/I)} \left(I^sR_P \cap R\right).
\]
There is also a notion of symbolic powers using minimal primes:
\[
\up{m}I^{(s)}=\bigcap_{P \in \Min(R/I)} \left(I^sR_P \cap R\right).
\]
Denote  $L={}\up{m}I^{(1)}$, then we have equalities $\Ass(R/L)=\Min(R/L)=\Min(R/I)$. From this, it is not hard to see that for all $s\ge 1$, the equality $\up{m}I^{(s)} = L^{(s)}$ holds. Hence the notion of ``associated'' symbolic power is more general than the notion of ``minimal'' symbolic powers.

The following lemma is folklore. We include an easy proof for completeness.
\begin{lem}
\label{lem_symbolic_powers}
Let $R$ be a noetherian ring, $I$ a proper ideal of $R$ such that $\Ass(I)=\Min(I)$. Let $I=Q_1\cap \cdots \cap Q_d$ be an irredundant primary decomposition of $I$. Then for all $s\ge 1$, there is an equality $I^{(s)}=Q_1^{(s)} \cap \cdots \cap Q_d^{(s)}$.
\end{lem}
\begin{proof}
Denote $P_i=\sqrt{Q_i}$, $i=1,\ldots,d$. Since each $P_i$ is a minimal prime of $I$, $IR_{P_i}=Q_iR_{P_i}$ for $1\le i\le d$. So
\[
I^{(s)}=\bigcap_{i=1}^d \left(I^sR_{P_i}\cap R \right)= \bigcap_{i=1}^d \left(Q_i^sR_{P_i}\cap R\right)=\bigcap_{i=1}^d Q_i^{(s)},
\]
which is the desired conclusion.
\end{proof}

We will use the following expansion for symbolic powers of sums several times.
\begin{lem}[H\`a--Jayanthan--Kumar--Nguyen {\cite[Theorem 4.1]{HJKN}}]
\label{lem_binomialformula}
 Let $R, S$ be noetherian algebras over a field $\kk$ such that $T=R\otimes_\kk S$ is also noetherian. Let $I \subseteq R, J \subseteq S$ be nonzero proper ideals. Then, for any integer $s \ge 1$, there is an equality $(I+J)^{(s)} = \sum\limits_{i=0}^s I^{(i)}J^{(s-i)}$.
\end{lem}

The next lemma, which is perhaps folklore, will be useful for the proof of \Cref{thm_depth_reg}. 
\begin{lem}
\label{lem_positivedepth_symbpow}
Let $R$ be a standard graded $\kk$-algebra, and $I$ a proper homogeneous ideal. 
\begin{enumerate}[\quad \rm (i)]
\item If $\depth(R/I)=0$, then for each $s\ge 1$, there is an equality 
$I^{(s)}=I^s.$
\item If $\depth(R/I)\ge 1$, then for each $s\ge 1$, the inequality
$\depth(R/I^{(s)})\ge 1$ holds.
\end{enumerate}
\end{lem}
\begin{proof}
From \cite[Lemma 2.2]{HJKN}, given any irredundant primary decomposition $I^s=Q_1\cap Q_2\cap \cdots \cap Q_d$, there are equalities
\begin{align*}
I^{(s)} &= \bigcap_{\sqrt{Q_i}\subseteq P \, \text{for some $P\in \Ass_R(I)$}}Q_i,\\
\Ass_R I^{(s)}&=\{\pp \in \Ass_R I^s \mid \text{$\pp$ is contained in an element of $\Ass_R(I)$}\}. 
\end{align*}
The desired conclusions follow.
\end{proof}
Recall that a standard graded $\kk$-algebra $(R,\mm)$ is called \emph{Koszul}, if $\reg_R(R/\mm)=0$. The following statement is well-known, and can be proved by standard short exact sequence arguments. Note that the equality of depth is due to Lescot. 
\begin{lem}
\label{lem_decomposition_F}
Let $(R,\mm)$ and $(S,\nn)$ be standard graded algebras over $\kk$ such that $\mm\neq (0)$ and $\nn\neq (0)$. Let $I\subseteq \mm^2, J\subseteq \nn^2$ be homogeneous ideals. Denote $T=R\otimes_\kk S$ and $F=I+J+\mm\nn$ the fiber product of $I$ and $J$.
\begin{enumerate}[\quad \rm (i)]
 \item There is an equation $F=(I+\nn)\cap (J+\mm)$. In particular, there is an exact sequence of $T$-modules
\[
0\to \frac{T}{F} \to \frac{R}{I} \oplus \frac{S}{J} \to \kk \to 0.
\]
\item There is an equality $\depth(T/F) = \min\left\{1, \depth(R/I), \depth(S/J)\right\}$.
\item Assume additionally that $R$ and $S$ are Koszul algebras. Then there is an equality
$\reg_T F  =\max\{2, \reg_R I, \reg_S J\}$.

If moreover either $I$ or $J$ is non-zero, then the last formula reduces to
$\reg_T F=\max\{\reg_R I, \reg_S J\}$.
\end{enumerate}
\end{lem}
\begin{proof}
Argue similarly to the proof of \cite[Proposition 3.3]{NgV19b}.
\end{proof}

\begin{rem}
\Cref{lem_decomposition_F}(iii) corrects two unfortunate (albeit minor) errors in \cite[Proposition 3.3]{NgV19b}: 
\begin{enumerate}
\item In the hypothesis of \emph{ibid.}, we need to assume that $\mm\neq (0)$ and $\nn\neq (0)$.
\item In the last statement of \emph{ibid.}, we need to assume that either $I$ or $J$ is non-zero.
\end{enumerate}
In fact, for (1), if $\mm=(0)$, then $R=\kk$, $I=(0)$. Hence we see that $T=S$, $F=J$, and the formula $\depth(T/F) = \min\left\{1, \depth(R/I), \depth(S/J)\right\}$
becomes
\[
\depth(S/J)=\min\left\{1, 0, \depth(S/J)\right\}=0
\]
which is incorrect if $\depth(S/J)\ge 1$. Similarly, we need $\nn\neq (0)$.

For (2), if $I=J=(0)$, and $R=\kk[x], S=\kk[y]$, then $F=(xy)$, hence
\[
\reg_T F=2 > \max\{\reg_R I, \reg_S J\}=-\infty.
\]
Thus we need the assumption  either $I$ or $J$ is non-zero.
\end{rem}

For the rest of this section, let $(R,\mm)$, $(S,\nn)$ be noetherian standard graded polynomial rings over $\kk$ of positive Krull dimensions, and $I, J$ proper homogeneous ideals of $R, S$, resp. Let $T=R\otimes_\kk S$, and $F=I+J+\mm\nn \subseteq T$ the fiber product of $I$ and $J$.

\begin{lem}
\label{lem_decompose_ordinarypow}
For every $s\ge 1$, there is an equality $F^s=I^s+J^s+\mm \nn F^{s-1}$. In particular,
\[
F^s=I^s+J^s+\mm\nn(I^{s-1}+J^{s-1})+\cdots+(\mm\nn)^s.
\]
\end{lem}
\begin{proof}
For the first assertion, it is harmless to assume $s\ge 2$. Since $F=I+J+\mm \nn$, we get
\[
F^s=(I+J)^s+\mm \nn F^{s-1}= I^s+J^s+\mm \nn F^{s-1}+\sum_{i=1}^{s-1}I^iJ^{s-i}.
\]
For $1\le i\le s-1$,
\[
I^iJ^{s-i}=IJ I^{i-1}J^{s-i-1}\subseteq \mm^2 \nn^2 I^{i-1}J^{s-i-1}=\mm\nn(\mm\nn I^{i-1}J^{s-i-1})\subseteq \mm \nn F^{s-1}.
\]
Hence the first equality holds true. The remaining equality is an immediate consequence.
\end{proof}

We record the following formulas for later use.
\begin{lem}[H\`a--Trung--Trung {\cite[Proposition 2.9]{HTT16}}]
\label{lem_depthregIplusn_ordinary}
For all $s\ge 1$, there are equalities:
\begin{enumerate}[\quad \rm (i)]
 \item $\depth \dfrac{T}{(I+\nn)^s}=\min\limits_{i\in [1,s]} \left\{ \depth \dfrac{R}{I^i} \right\}$.
 \item $\reg \dfrac{T}{(I+\nn)^s}=\max \limits_{i\in [1,s]} \left\{ \reg \dfrac{R}{I^i}+s-i \right\}$.
\end{enumerate}
\end{lem}

What we will also need is the following analogous formulas for symbolic powers.
\begin{lem}
\label{lem_depthregIplusn_symbolic}
For all $s\ge 1$, there are equalities:
\begin{enumerate}[\quad \rm (i)]
 \item $\depth \dfrac{T}{(I+\nn)^{(s)}}=\min\limits_{i\in [1,s]} \left\{ \depth \dfrac{R}{I^{(i)}} \right\}$.
 \item $\reg \dfrac{T}{(I+\nn)^{(s)}}=\max \limits_{i\in [1,s]} \left\{ \reg \dfrac{R}{I^{(i)}}+s-i \right\}$.
\end{enumerate}
\end{lem}
\begin{proof}

(i) By induction on $n=\dim S$, we reduce to the case $n=1$, namely $S=\kk[y]$, $\nn=(y)$, and $T=R[y]$. We have to show that
\[
\depth \frac{T}{(I+(y))^{(s)}} =\min_{i\in [1,s]} \left\{ \depth \frac{R}{I^{(i)}} \right\}.
\]
We argue similarly to \cite[Proof of Theorem 5.2]{CH+}. Thanks to \Cref{lem_binomialformula}, 
\[
(I+(y))^{(s)}=I^{(s)}+I^{(s-1)}y+\cdots+Iy^{s-1}+(y^s).
\]
Hence we have a direct sum decompositions of finitely generated $R$-modules
\[
\frac{T}{(I+(y))^{(s)}} \cong \bigoplus_{i=1}^s \frac{R}{I^{(i)}}y^{s-i}.
\]
The desired conclusion immediately follows from this decomposition.

(ii) The proof is similar to part (i), and is left to the interested reader.
\end{proof}


\section{Decompositions as intersections and sums}
\label{sect_decomposition}

From now on, we keep the following notations.
\begin{notn}
\label{notn_IandJ}
Let $\kk$ be a field of arbitrary characteristic. 
\begin{itemize}
 \item Let $R$ and $S$ be noetherian standard graded polynomial rings over $\kk$, such that $\dim R\ge 1$ and $\dim S\ge 1$.
 \item Let $T=R\otimes_{\kk}S$ be the tensor product over $\kk$ of $R$ and $S$. 
 \item The homogeneous maximal ideals of $R$ and $S$ are $\mathfrak{m}$ and $\mathfrak{n}$, respectively. 
 \item Let $I\subseteq \mm^2, J\subseteq \nn^2$ be proper homogeneous ideals of $R$ and $S$, resp. 
 \item Let $F=I+J+\mm \nn \subseteq T$ be the fiber product of $I$ and $J$.
\end{itemize}
\end{notn}
The main results of this section are the following two decomposition formulas for the symbolic power of fiber products.
\begin{lem} 
\label{lem_decomposition}
Keep using \Cref{notn_IandJ}.
\begin{enumerate}[\quad \rm (1)]
 \item Assume that $\min\left\{\depth(R/I),\depth(S/J)\right \}\ge 1$.  Then for every integer $s\geq 1$, we have an equality 
$$
F^{(s)}=(I+\nn)^{(s)}\cap(J+\mm)^{(s)}.
$$
\item Assume that $\min\left\{\depth(R/I),\depth(S/J)\right \}=0$. Then for every integer $s\ge 1$, there is an equality between the $s$-th symbolic and ordinary powers
\[
F^{(s)}=F^s.
\]
\end{enumerate}
\end{lem}

\begin{lem}
\label{lem_decomposition_as_sum}
Assume that $\min\left\{\depth(R/I),\depth(S/J)\right \}\ge 1$. For all $s\ge 1$, there is an equality
\[
F^{(s)}= \sum_{i=0}^{s}\sum_{t=0}^{s}(I^{(i)}\cap \mm^{s-t})(J^{(t)}\cap \nn^{s-i}).
\]
\end{lem}

We begin with the description of the associated (minimal) primes of $F$ in terms of the associated (respectively, minimal) primes of the components $I$ and $J$.
\begin{lem}
\label{lem_AssMin_F}
The following statements hold.
\begin{enumerate}[\quad \rm (i)]
\item 
$\Ass_T F = \{\pp_1+\nn: \pp_1 \in \Ass_R(I)\} ~ \bigcup ~ \{\pp_2+\mm: \pp_2 \in \Ass_S(J)\}.$
\item 
$\Min_T F \subseteq  \{\pp_1+\nn: \pp_1 \in \Min_R(I)\} ~ \bigcup ~ \{\pp_2+\mm: \pp_2 \in \Min_S(J)\}.$
The equality happens if both $\dim(R/I)$ and $\dim(S/J)$ are positive.
\item Assume that $\dim(R/I)=0$. Then there is an equality
$$\Min_T F = \{\pp_2+\mm: \pp_2 \in \Min_S(J)\}.$$
\end{enumerate}
\end{lem}
\begin{proof}

(i) From the exact sequence of $T$-modules
\[
0\to \frac{T}{F} \to \frac{T}{I+\nn} \oplus \frac{T}{J+\mm} \to \kk \to 0,
\]
and \Cref{lem_Ass_tensor}, we deduce
\begin{align*}
\Ass_T F & \subseteq \Ass_T \frac{T}{I+\nn} \cup \Ass_T \frac{T}{J+\mm} \\
        &= \{\pp_1+\nn: \pp_1 \in \Ass_R(I)\} ~ \bigcup ~ \{\pp_2+\mm: \pp_2 \in \Ass_S(J)\}.
\end{align*}

It remains to prove the reverse containment. Take $\pp_1\in \Ass_R I$ and let $P=\pp_1+\nn$. If $\pp_1=\mm$, then $\depth(R/I)=0$. Thanks to \Cref{lem_decomposition_F},
\[
\depth(T/F)=\min\{1,\depth(R/I),\depth(S/J)\}=0.
\]
In particular, $P=\mm+\nn \in \Ass_T F$.

If $\pp_1 \neq \mm$, we have $F_P=(I+J+\nn)_P=(I+\nn)_P$. By Lemma \ref{lem_Ass_tensor} and the fact that $T/(I+\nn) \cong (R/I)\otimes_\kk (S/\nn)$, we get $P=\pp_1+\nn \in \Ass_T(I+\nn)$. Hence
\[
\depth (T/F)_P=\depth (T/(I+\nn))_P=0.
\]
Therefore $P \in \Ass_T F$, as desired.

(ii) As for part (i), the first containment follows from the exact sequence
\[
0\to \frac{T}{F} \to \frac{T}{I+\nn} \oplus \frac{T}{J+\mm} \to \kk \to 0,
\]
and \Cref{lem_Ass_tensor}.

Now assume that $\min\{\dim(R/I),\dim(S/J)\}\ge 1$. Take $P=\pp_1+\nn$, where $\pp_1\in \Min_R(I)$. By Lemma \ref{lem_Ass_tensor}, $P\in \Min_T(I+\nn)$. Since $\dim(R/I)>0$, $\mm \not\subseteq P$, so 
$$
\dim (T/F)_P=\dim (T/(I+\nn))_P=0.
$$
Therefore, $P\in \Min_T F$, as claimed. Similar arguments work when $P=\pp_2+\mm$, where $\pp_2\in \Min_S(J)$.

(iii) Note that $F=(I+\nn)\cap (J+\mm)$. Since $\dim(R/I)=0$, $(I+\nn)$ is $(\mm+\nn)$-primary. Hence a prime ideal contains $F$ if and only it contains $J+\mm$. Consequently, $\Min_T F=\Min_T (J+\mm)$. The desired conclusion then follows by applying Lemma \ref{lem_Ass_tensor}.
\end{proof}
Now we are ready to present the
\begin{proof}[Proof of \Cref{lem_decomposition}]
(1) We wish to show that for all $s\ge 1$, the following holds
$$
F^{(s)}=(I+\nn)^{(s)}\cap(J+\mm)^{(s)}.
$$
The case $s=1$ is a consequence of \Cref{lem_decomposition_F}: We have
\[
F^{(1)}=F=I+J+\mm \nn=(I+\nn) \cap (J+\mm) = (I+\nn)^{(1)}\cap(J+\mm)^{(1)}.
\]
Since $\depth(R/I), \depth(S/J)\ge 1$, we deduce from \Cref{lem_decomposition_F} that $\depth(T/F)\ge 1$. We have
\begin{align*}
F^{(s)} &=\bigcap_{P\in \Ass_T(F)} (F^sT_P \cap T)\\
        &= \bigcap_{\pp_1\in \Ass_R(I)} (F^sT_{\pp_1+\nn} \cap T) \quad \bigcap \quad  \bigcap_{\pp_2\in \Ass_S(J)} (F^sT_{\pp_2+\mm} \cap T).
\end{align*}
It remains to show that
\begin{align*}
 \bigcap_{\pp_1\in \Ass_R(I)} (F^sT_{\pp_1+\nn} \cap T) &= (I+\nn)^{(s)},\\
 \bigcap_{\pp_2\in \Ass_S(J)} (F^sT_{\pp_2+\mm} \cap T) &=(J+\mm)^{(s)}.
\end{align*}
We prove the first equality; the second one is similar. For any $\pp_1\in \Ass_R(I)$, as $\depth(R/I)>0$, $\pp_1\neq \mm$. Hence $\mm \not\subseteq \pp_1+\nn$, so we get the second equality in the chain
\[
F^sT_{\pp_1+\nn}=(I+J+\mm \nn)^s_{\pp_1+\nn}=(I+J+\nn)^s_{\pp_1+\nn}=(I+\nn)^s_{\pp_1+\nn}.
\]
Thanks to \Cref{lem_Ass_tensor}, $\Ass_T(I+\nn)=\{\pp_1+\nn \mid \pp_1 \in \Ass_R(I)\}$. Hence
\begin{align*}
 \bigcap_{\pp_1\in \Ass_R(I)} (F^sT_{\pp_1+\nn} \cap T) &= \bigcap_{\pp_1\in \Ass_R(I)} ((I+\nn)^sT_{\pp_1+\nn} \cap T) \\
                                                        &=\bigcap_{P\in \Ass_T(I+\nn )} ((I+\nn)^sT_P \cap T) =(I+\nn)^{(s)},
\end{align*}
as desired. This concludes the proof of part (1).

(2) Without loss of generality, assume that $\depth(R/I)=0$. Then thanks to \Cref{lem_AssMin_F}, $\mm+\nn \in \Ass_T(F)$, namely $\depth(T/F)=0$. Thanks to \Cref{lem_positivedepth_symbpow}, the last equation implies $F^{(s)}=F^s$ for all $s$, as desired.
\end{proof}
\begin{ex}
In general, the decomposition formula 
$$F^{(s)}=(I+\nn)^{(s)}\cap(J+\mm)^{(s)}
$$ 
need not hold without the assumption $\min \{\depth(R/I),\depth(S/J) \}\ge 1$. For example, let $R=\kk[x], I=(x^2), S=\kk[y], J=(y^2)$. Then $F=(x^2,y^2,xy)$. Hence
\[
x^2y\in (I+(y))^{(2)} \cap (J+(x))^{(2)}=(x^2,y)^2 \cap (x,y^2)^2.
\]
On the other hand, $x^2y\notin F^{(2)}=F^2=(x,y)^4$, by degree reason. Thus $F^{(2)}\subsetneq (I+\nn)^{(2)}\cap(J+\mm)^{(2)}$ in this case.

\end{ex}

A \emph{filtration} of ideals in $R$ is a descending chain $K_\bullet=(K_i)_{i\ge 0}$ consisting of ideals of $R$ satisfying
\[
K_0 \supseteq K_1 \supseteq K_2 \supseteq \cdots 
\]
For example, if $I$ is a homogeneous ideal of $R$, then the ordinary powers $I^\bullet$ and the symbolic powers $I^{(\bullet)}$ are filtration of homogeneous ideals of $R$. \Cref{lem_decomposition_as_sum} is a consequence of the following more general statement.
\begin{lem}
\label{lem_intersect_formula}
Let $K_\bullet$ be a filtration of homogeneous ideals in $R$ and $L_\bullet$ be a filtration of homogeneous ideals in $S$. Then for all $s\ge 1$, there is an equality
\begin{equation}
\label{eq_decompose_intersect}
\left(\sum_{i=0}^s K_i\nn^{s-i}\right) \bigcap \left(\sum_{t=0}^s L_t\mm^{s-t}\right) =\sum_{i=0}^s\sum_{t=0}^s (K_i\cap \mm^{s-t})(L_t\cap \nn^{s-i}).
\end{equation}
\end{lem}
\begin{proof}

For each $0\le i,t \le s$, denote $G_{i,t}=(K_i \cap \mm^{s-t})(L_t\cap \nn^{s-i})$, so the right-hand side of \eqref{eq_decompose_intersect} is nothing but $\sum\limits_{i,t=0}^s G_{i,t}$. Clearly the left-hand side contains the right-hand side.

It remains to prove the reverse containment. For this, consider the $\Z^2$-grading of $T$ given by $\deg x_i=(1,0), \deg y_j=(0,1)$ for $1\le i\le m, 1\le j\le n$. Then the ideals of the filtration $K_\bullet$ and $L_\bullet$ are graded in the $\Z$-gradings of $R$ and $S$, they are also bigraded. Thus both sides of  \eqref{eq_decompose_intersect} are bigraded. Thus it remains to show that for any bigraded element $x$ of the left-hand side of \eqref{eq_decompose_intersect} also belongs to the right-hand side.

 Since $K_\bullet$ is a filtration, $x\in \sum_{i=0}^s K_i\nn^{s-i}\subseteq K_0T$. Similarly, $x\in L_0T$.

Let $\deg x=(a,b)$, where $a,b\ge 0$. By symmetry, it is harmless to assume that $a\le b$. 

Consider the following cases.

\textbf{Case 1:} $a\ge s$. In this case $x\in \mm^s \cap \nn^s$. Hence per flatness and \Cref{lem_intersection}, we get the first and second equality in the chain
\begin{align*}
x\in K_0T\cap \mm^sT\cap L_0T\cap \nn^sT & =(K_0\cap \mm^s)T \cap (L_0\cap \nn^s)T\\
                                         & =(K_0\cap \mm^s)(L_0\cap \nn^s)=G_{0,0}. 
\end{align*}

\textbf{Case 2:} $a\le s-1$, $b\ge s$. We claim that $x\in \mm^aL_{s-a}$. Indeed, since $L_\bullet$ is a filtration,
\[
x\in \sum_{t=0}^s L_t\mm^{s-t} \subseteq L_{s-a}+\mm^{a+1}.
\]
Since $x, L_{s-a}, \mm^{a+1}$ are bigraded, this yields an equation $x=x'+x''y$, where $x'\in L_{s-a}, x''\in \mm^{a+1}, y\in T$ are bigraded elements. As $\deg x=(a,b)$, this implies $x''y=0, x=x'\in L_{s-a}$. Since $\deg x=(a,b)$, clearly $x\in \mm^a$, and  hence
\[
x\in \mm^a \cap L_{s-a}=\mm^aL_{s-a}.
\]
thanks to \Cref{lem_intersection}. This is the desired claim.

As $b\ge s$, $x\in \nn^s$, and as above, $x\in K_0T$. Hence thanks to \Cref{lem_intersection},
\[
x\in K_0T\cap \mm^aT\cap (L_{s-a}\cap \nn^s)T=(K_0\cap \mm^a)(L_{s-a}\cap \nn^s)= G_{0,s-a}.
\]

\textbf{Case 3:} $a\le s-1$, $b\le s-1$. Arguing as in Case 2, we get $x\in \mm^aL_{s-a}$ and similarly, that $x\in \nn^bK_{s-b}$. In particular, using \Cref{lem_intersection} again,
\[
x\in (K_{s-b}\cap \mm^a) \cap (L_{s-a}\cap \nn^b)=(K_{s-b}\cap \mm^a) (L_{s-a}\cap \nn^b)=G_{s-b,s-a}.
\]
In any case, $x$ belongs to the right-hand side of \eqref{eq_decompose_intersect}. The proof is concluded.
\end{proof}

It remains to present the
\begin{proof}[Proof of \Cref{lem_decomposition_as_sum}]
As $\min\left\{\depth(R/I),\depth(S/J)\right \}>0$, \Cref{lem_decomposition} yields
\begin{equation*}
F^{(s)}= (I+\nn)^{(s)} \cap (J+\mm)^{(s)}=\left(\sum_{i=0}^s I^{(i)}\nn^{s-i}\right)\bigcap  \left(\sum_{t=0}^s J^{(t)}\mm^{s-t}\right). 
\end{equation*}
The second equality follows from \Cref{lem_binomialformula} and the fact that symbolic and ordinary powers coincide for each of $\mm$ and $\nn$.

Applying \Cref{lem_intersect_formula} for the chains of homogeneous ideals $I^{(\bullet)}$ and $J^{(\bullet)}$, we get the desired equality.
\end{proof}


\section{Depth and regularity of symbolic powers}
\label{sect_symbpow_positivedepth}

Keep using \Cref{notn_IandJ}. Denote $\pp:=\mm T+\nn T$ the graded maximal ideal of $T$. The main result of this section determines depth and regularity of symbolic powers of fiber products, in the case both $\depth(R/I)$ and $\depth(S/J)$ are positive. 

\vskip6ex

\begin{thm} 
\label{thm_depth_reg}
Assume that $\min\left\{\depth(R/I),\depth(S/J)\right \}\ge 1$. Then for every integer $s\geq 1$, there are equalities
\begin{enumerate}[\quad \rm (i)]
\item $\depth(T/F^{(s)})=1$, and
\item $\reg(F^{(s)})=\max\limits_{i\in [1,s]} \left \{2s,\reg(I^{(i)})+s-i, \reg(J^{(i)})+s-i \right \}$.
\end{enumerate}
\end{thm}

The main ingredients in the proof are Lemmas \ref{lem_positivedepth_symbpow}, \ref{lem_depthregIplusn_symbolic}, and the following 
\begin{lem}
\label{lem_mingen_deg2s}
Let $K_\bullet, L_\bullet$ be filtration of homogeneous ideals of $R,S$, respectively. For each $s\ge 1$, denote
\[
W_s=\left(\sum_{i=0}^s K_i\nn^{s-i}\right) \bigcap \left(\sum_{t=0}^s L_t\mm^{s-t}\right).
\]
Assume that the following conditions are simultaneously satisfied: 
\begin{enumerate}[\quad \rm (1)]
 \item $\mm \subseteq K_0, \nn\subseteq L_0$; and,
 \item $\min\left\{\dim(R/K_1),\dim(S/L_1)\right\}>0$.
\end{enumerate}
Then for all $s\ge 1$, $\mm^s\nn^s\not\subseteq \pp W_s$. In particular, $W_s$ has a minimal homogeneous generator of degree $2s$.
\end{lem}
\begin{proof}
By \Cref{lem_intersect_formula}, there are equalities
\[
W_s=\sum_{0\le i,t\le s}(K_i\cap \mm^{s-t})(L_t \cap \nn^{s-i})=\mm^s\nn^s+\mathop{\sum_{0\le i,t\le s}}_{(i,t)\neq (0,0)} (K_i\cap \mm^{s-t})(L_t\cap \nn^{s-i}).
\]
The second equality holds since $K_0\supseteq \mm \supseteq \mm^s$ and $L_0\supseteq \nn \supseteq \nn^s$.

If the first assertion is true, then some minimal generator of $\mm^s\nn^s$ is a minimal generator of $W_s$, namely the latter has a minimal homogeneous generator of degree $2s$. Thus it remains to prove the first assertion.

Assume the contrary, that $\mm^s\nn^s\subseteq \pp W_s$. Then together with Nakayama's lemma, the last display implies that
\[
\mm^s\nn^s\subseteq \mathop{\sum_{0\le i,t\le s}}_{(i,t)\neq (0,0)} (K_i\cap \mm^{s-t})(L_t \cap \nn^{s-i}) \subseteq K_1+L_1.
\]
With respect to the standard bigrading of $T=R\otimes_\kk S$, both sides of the containment $\mm^s\nn^s\subseteq K_1+L_1$ are bigraded. Take any bigraded element $x\in \mm^s\nn^s$, and minimal homogeneous generators $f_1,\ldots,f_k$ of $K_1$ and $g_1,\ldots,g_l$ of $L_1$. Then $\deg x=(a,b)$, where $a,b\ge s$, $\deg f_i=(a_i,0)$, $\deg g_j=(0,b_j)$, where $a_1,\ldots,a_k,b_1,\ldots,b_l\ge 0$. The fact that $x\in K_1+L_1$ implies the existence of bigraded elements $u_1,\ldots,u_k,$  $v_1,\ldots,v_l\in T$,  such that $\deg u_i=(a-a_i,b), \deg v_j=(a,b-b_j)$ and
\[
x=u_1f_1+\cdots+u_kf_k+v_1g_1+\cdots+v_lg_l.
\]
Since $f_i\in K_1, g_j\in L_1, a,b\ge s$, the last equation shows that $x\in (K_1\cap \mm^s)\nn^s+(L_1\cap \nn^s)\mm^s$. Therefore
\[
\mm^s\nn^s \subseteq (K_1\cap \mm^s)\nn^s+(L_1\cap \nn^s)\mm^s \subseteq \mm^s\nn^s
\]
so equalities hold from left to right. Now
\[
\frac{\mm^s}{K_1\cap \mm^s}\otimes_\kk \frac{\nn^s}{L_1\cap \nn^s}= \frac{\mm^s\nn^s}{(K_1\cap \mm^s)\nn^s+(L_1\cap \nn^s)\mm^s}=0,
\]
so either $\mm^s\subseteq K_1$ or $\nn^s\subseteq L_1$. This contradicts the hypothesis that $\dim(R/K_1)$, $\dim(S/L_1)>0$. So the first assertion holds true and the proof is concluded.
\end{proof}

\begin{proof}[Proof of \Cref{thm_depth_reg}]
Set $I':=I+\mathfrak{n}$ and $J':=J+\mathfrak{m}$. 

(i) Using Lemma \ref{lem_decomposition}, we obtain the following short exact sequence
\begin{equation}
\label{eq_exactseq_Fs}
0 \longrightarrow T/F^{(s)}\longrightarrow T/I'^{(s)}\oplus T/J'^{(s)}\longrightarrow T/(I'^{(s)}+J'^{(s)})\longrightarrow 0.
\end{equation}
We prove the following

\textbf{Claim}: The following hold true.
\begin{enumerate}
 \item $\dfrac{T}{I'^{(s)}+J'^{(s)}}$ is artinian, hence has depth 0,
 \item $\min\left\{\depth T/I'^{(s)}, \depth T/J'^{(s)} \right\}\ge 1$.
\end{enumerate}
Together with the exact sequence \eqref{eq_exactseq_Fs} and the depth lemma, this claim implies the desired equation $\depth(T/F^{(s)})=1$.

\emph{Proof of the Claim:}

(1) Per \Cref{lem_binomialformula}, 
\[
I'^{(s)} =\sum_{i=0}^s I^{(i)}\nn^{s-i} \supseteq \nn^s,
\]
hence $\mm^s+\nn^s \subseteq I'^{(s)}+J'^{(s)}$. In particular, $\dfrac{T}{I'^{(s)}+J'^{(s)}}$ is Artinian, and thus has depth 0.

(2) Thanks to \Cref{lem_depthregIplusn_symbolic}, we obtain the second equality in the following chain
\[
\depth \frac{T}{I'^{(s)}}= \depth \frac{T}{(I+\nn)^{(s)}}=\min_{i\in [1,s]} \left\{ \depth \frac{R}{I^{(i)}} \right\} \ge 1.
\]
The inequality follows from the hypothesis $\depth(R/I)>0$ and \Cref{lem_positivedepth_symbpow}. Similarly $\depth T/J'^{(s)} \ge 1$, finishing the proof of the claim and that of part (i).

(ii) If $I=(0)$ and $J=(0)$ then $F^s=\mm^s\nn^s$ has regularity $2s$ thanks to Lemma \ref{lem_tensor}. Without loss of generality, we may suppose that $I\neq (0)$. We must prove that
$$\reg (F^{(s)})=\max_{i,j \in [1,s]} \{2s,\reg(I^{(i)})+s-i, {\rm reg}(J^{(j)})+s-j\}.$$
As mentioned in the proof of (i), $T/(I'^{(s)}+J'^{(s)})$ is an Artinian ring. Hence, by \cite[Theorem 18.4]{Pe11}, the regularity of $T/(I'^{(s)}+J'^{(s)})$ is the maximum degree of a monomial in $T$ which does not belong to $I'^{(s)}+J'^{(s)}$. For any monomial $v\in T$, with ${\rm deg}(v)\geq 2s-1$, we have $v\in \mathfrak{m}^s+\mathfrak{n}^s\subseteq I'^{(s)}+J'^{(s)}.$ Therefore,
\[
{\rm reg}\big(T/(I'^{(s)}+J'^{(s)}\big)\leq 2s-2.
\]

On the other hand, $F^{(2s)}$ has a minimal generator of degree $2s$ by Lemma \ref{lem_mingen_deg2s}. Hence
\begin{equation}
\label{eq_compare_reg}
\reg T/F^{(s)} \ge 2s-1 > \reg T/(I'^{(s)}+J'^{(s)}).
\end{equation}
The exact sequence \eqref{eq_exactseq_Fs} implies the first equality in the following chain
\begin{align*}
\reg T/F^{(s)} &= \max \left\{\reg T/(I'^{(s)}+J'^{(s)})+1, \reg (T/I'^{(s)}), \reg (S/J'^{(s)}) \right\}\\
               &= \max \left \{2s-1, \reg (T/I'^{(s)}), \reg (S/J'^{(s)}) \right \}\\
               &= \max\limits_{i\in [1,s]} \left \{2s-1, \reg (R/I^{(i)})+s-i, \reg (S/J^{(i)})+s-i \right \}
\end{align*}
The second equality holds by using \eqref{eq_compare_reg}; the third one follows from \Cref{lem_depthregIplusn_symbolic}(ii). From the last chain, we get the desired equality, and conclude the proof.
\end{proof}

The following consequence of Theorem \ref{thm_depth_reg} recovers \cite[Theorem 3.2]{KKS} and \cite[Corollary 4.12]{OR19}, which require that both $I$ and $J$ are squarefree monomial ideals.
\begin{cor}
\label{cor_reg_Fs_unmixedcase}
Assume that $\min\left\{\depth(R/I),\depth(S/J)\right \}\ge 1$. Assume moreover that $I$ is a non-zero monomial ideal of $R$ that is unmixed, namely $\Ass_R(I)=\Min_R(I)$. Then for all $s\ge 1$, there is an equality
$$
\reg (F^{(s)})=\max \limits_{i\in [1,s]} \left\{\reg(I^{(i)})+s-i, \reg(J^{(i)})+s-i \right\}.
$$
\end{cor}
The proof of this corollary requires \Cref{lem_maxdegree_symbolicpowers}, which provides a lower bound for the regularity of symbolic powers of unmixed monomial ideals.   If $I$ is a monomial ideal, denote by $\Gcc(I)$ the set of minimal monomial generators of $I$.
\begin{lem}
\label{lem_maxdegree_symbolicpowers}
Let $I$ be a non-zero proper unmixed monomial ideal of $R$. Let $\Gcc(\sqrt{I})=\{f_1,\ldots,f_k\}$. The following statements hold.
\begin{enumerate}[\quad \rm(i)]
 \item For every $s\ge 1$ and every $1\le i\le k$, there exists an element $g_i\in \Gcc(I^{(s)})$ that is divisible by $f_i^s$.
 \item  Assume further that $I\subseteq \mm^2$. Then for every $s\ge 1$, there is an inequality $d(I^{(s)}) \ge \max\{2,d(\sqrt{I})\}s$.
\end{enumerate}
\end{lem}
\begin{proof}
 (i) Denote by $x_1,\ldots,x_m$ the variables of $R$. It suffices to consider $i=1$, and without loss of generality, assume that $f_1=x_1x_2\cdots x_t$, where $t\ge 1$. Write $f=f_1$ for simplicity. We claim that for every $s\ge 1$, there is an element $g\in I^{(s)}$ which is divisible by $f^s$.
 
 Since $f\in \sqrt{I}$, for some $q\ge 1, f^q\in I$. Hence $f^{qs} \in I^s \subseteq I^{(s)}$. Thus there exist non-negative integers $\alpha_1,\ldots,\alpha_t$ such that $g=x_1^{\alpha_1}\cdots x_t^{\alpha_t}\in \Gcc(I^{(s)})$. We claim that $\alpha_i\ge s$ for all $i=1,\ldots,t$.
 
 Assume the contrary, that for instance $\alpha_1\le s-1$. Let $I=Q_1\cap \cdots \cap Q_d$ be an irredundant primary decomposition of $I$, where the components $Q_i$ are monomial ideals. Note that $Q_i^s$ is primary for every $i=1,\ldots,d$. Hence by the hypothesis that $I$ is unmixed and Lemma \ref{lem_symbolic_powers}, $I^{(s)}=Q_1^s\cap \cdots \cap Q_d^s$.
 
For each $1\le j\le d$, we have $g=x_1^{\alpha_1}\cdots x_t^{\alpha_t}\in Q_j^s$. Hence $g$ is a product of $s$ monomials in $Q_j$. Since $\alpha_1\le s-1$, one of these monomials, say $h_j$, divides $g'=x_2^{\alpha_2}\cdots x_t^{\alpha_t}$. In other words, $g'\in Q_j$.

In particular, $g'\in Q_1\cap \cdots \cap Q_d=I$. This implies $x_2\cdots x_t\in \sqrt{I}$, which contradicts the assumption that $f=x_1x_2\cdots x_t$ is a minimal generator of $\sqrt{I}$.

Therefore $\alpha_i\ge s$ for all $i=1,\ldots,t$. This shows that $f^s$ divides $g\in \Gcc( I^{(s)})$, as desired.

(ii) It is harmless to assume that $\deg(f_1)=d(\sqrt{I})$. If $d(\sqrt{I})\ge 2$ holds then by part (i), $I^{(s)}$ has a minimal generator of degree at least $s\deg(f_1)=d(\sqrt{I})s \ge 2s$, so we are done. Assume that $d(\sqrt{I})=1$. By renaming the variables, we can assume that $f_i=x_i$ for every $i=1,\ldots,k$. Hence $I$ is a primary ideal with $\sqrt{I}=(x_1,\ldots,x_k)$. This implies that $I^{(s)}=I^s$, which is generated in degree at least $2s$, as $I\subseteq \mm^2$. Therefore $d( I^{(s)}) \ge 2s$. The proof is concluded.
\end{proof}
Next we present the

\begin{proof}[Proof of Corollary \ref{cor_reg_Fs_unmixedcase}]
By Theorem \ref{thm_depth_reg}, there is an equality
$$
\reg (F^{(s)})=\max\limits_{i\in [1,s]}\left\{2s,\reg(I^{(i)})+s-i, \reg(J^{(i)})+s-i \right\}.
$$
It remains to observe that by Lemma \ref{lem_maxdegree_symbolicpowers}, $\reg I^{(s)}\ge d(I^{(s)}) \ge 2s$.
\end{proof}

Denote $\Rcc_s(I)=R\oplus I^{(1)}t \oplus I^{(2)}t^2 \oplus \cdots \subseteq R[t]$ the \emph{symbolic Rees algebra} of $I$. The \emph{symbolic analytic spread} of $I$ is defined as
\[
\ell_s(I):= \dim(\Rcc_s(I)/\mm \Rcc_s(I)).
\]
It is known that if $\Rcc_s(I)$ is Noetherian, then $\ell_s(I)$ is finite. In particular, this is the case when $I$ is a monomial ideal. See \cite{DM21} for more discussion on the symbolic analytic spread. Let $I\subset R$ be a squarefree monomial ideal. It is known by \cite[Theorem 2.4]{HKTT} that $${\rm depth}\big(R/I^{(k)}\big)=\dim R-\ell_s(I),$$for any integer $k\gg 0$. Thus, we obtain the following corollary as a consequence of Theorem \ref{thm_depth_reg}.

\begin{cor}
Assume moreover that $I$ and $J$ are squarefree monomial ideals. Then there is an equality 
$$
\ell_s(F)=\dim R+ \dim S-1.
$$
\end{cor}

In the case $\min \left\{\depth(R/I),\depth(S/J)\right\}=0$, we will prove the following result.
\begin{thm}
\label{thm_depth_zerodepthcase}
Assume that $\min \left\{\depth(R/I),\depth(S/J)\right\}=0$. Then for all $s\ge 1$,  there are equalities $F^{(s)}=F^s$, and
\begin{align*}
\depth(T/F^{(s)})&=0,\\
\reg F^{(s)} &= \max \limits_{i\in [1,s]} \left\{\reg(\mm^{s-i}I^i)+s-i, \reg(\nn^{s-i}J^i)+s-i\right\}.
\end{align*}
\end{thm}
It is clear from \Cref{lem_decomposition} that in the non-trivial case where $\min \{\depth(R/I),$ $\depth(S/J)\}=0$, and $I$ and $J$ are not both zero, $\depth(T/F)=0$ and $F^{(s)}=F^s$ for all $s\ge 1$. Therefore, it is necessary to study the depth and regularity of ordinary powers of $F$ first. This is what we will do in the next section, before returning to the proof of \Cref{thm_depth_zerodepthcase}.

\section{Depth and regularity of ordinary powers}
\label{sect_ordinarypowers}

Keep using \Cref{notn_IandJ}. The following result is a sharpening of \cite[Theorem 6.1]{NgV19b}: we are able to relax the hypothesis on the characteristic of $\kk$ of \emph{ibid.}
\begin{prop}
\label{prop_depthordinarypower_0}
Let $\kk$ be a field of arbitrary characteristic. Assume that $I$ and $J$ are not both zero. Then for all $s\ge 2$, there is an equality $\depth(T/F^s)=0$.
\end{prop}
The following technical lemma plays a key role in this section.
\begin{lem}
\label{lem_noncontainmentsandcontainments}
Let $s\ge 2$ and $1\le i\le s-1$ be integers. Assume that $I\subseteq \mm^2$. Then the following statements hold.
\begin{enumerate}[\quad \rm (1)]
\item There is a containment between ideals of $T$
\begin{equation*}
(I^i\cap \mm^s)\nn^{s-i}\cap F^s \subseteq (\mm\nn)^{s-i}F^i. 
\end{equation*}
and an equality between ideals of $R$
\[
(F^s:_T\nn^{s-i}) \cap (I^i\cap \mm^s)R=\mm^{s-i}I^i.
\]
\item Assume moreover that $I\neq (0)$. Then there is a non-containment 
$$
I^i\mm^{s-i-1}\nn^{s-i}\not\subseteq F^s.
$$
\end{enumerate}
\end{lem}
\begin{proof}
(1) For the containment, we product by reverse induction on $i\le t\le s$ that
\[
(I^i\cap \mm^s)\nn^{s-i}\cap F^s \subseteq (\mm\nn)^{s-t}F^t.
\]
There is nothing to do if $t=s$. Assume that $i\le t\le s-1$ and the statement is true for $t+1$, so
\[
(I^i\cap \mm^s)\nn^{s-i}\cap F^s \subseteq (\mm\nn)^{s-t-1}F^{t+1}=(\mm\nn)^{s-t-1}I^{t+1}+(\mm\nn)^{s-t-1}J^{t+1}+(\mm\nn)^{s-t}F^t.
\]
The equality holds because of \Cref{lem_decompose_ordinarypow}.

All the ideals in the last display are bigraded in the standard bigrading of $T$. Take a bigraded element $x\in (I^i\cap \mm^s)\nn^{s-i}\cap F^s$. Then $\deg x=(a,b)$ where $a\ge s, b\ge s-i$. From the last display, inspecting degrees, we conclude that
\begin{align*}
x&\in \mm^{s-t-1}I^{t+1}\cap \nn^{s-i}+\mm^s\cap (\nn^{s-t-1}J^{t+1})+(\mm\nn)^{s-t}F^t \\
  &=   \mm^{s-t-1}I^{t+1}\nn^{s-i}+\mm^s\nn^{s-t-1}J^{t+1}+(\mm\nn)^{s-t}F^t\\
  &=(\mm\nn)^{s-t-1}I\nn^{t+1-i}I^t+(\mm\nn)^{s-t-1}\mm^{t+1}J^{t+1}+(\mm\nn)^{s-t}F^t\\
  & \subseteq (\mm\nn)^{s-t}\nn^{t-i}I^t+(\mm\nn)^{s-t}\mm^tJ^t+(\mm\nn)^{s-t}F^t\\
  &=(\mm\nn)^{s-t}F^t.
\end{align*}
The first equality in the chain follows from \Cref{lem_intersection}. The containment $\subseteq$ holds since $I\subseteq \mm, J\subseteq \nn$. The last equality holds since $I,J\subseteq F$. Hence we complete the induction step, and thereby obtain
\begin{equation}
\label{eq_containmentinT}
(I^i\cap \mm^s)\nn^{s-i}\cap F^s \subseteq (\mm\nn)^{s-i}F^i.
\end{equation}

For the equality
\[
(F^s:_T\nn^{s-i}) \cap (I^i\cap \mm^s)R=\mm^{s-i}I^i,
\]
since $I,\mm\nn \subseteq F$, clearly the left-hand side contains the right-hand side. For the reverse containment, since both sides are homogeneous ideals of $R$, it suffices to take an arbitrary homogeneous element $x$ in the left-hand side. Then
\[
x\nn^{s-i} \in (I^i\cap \mm^s)\nn^{s-i}\cap F^s \subseteq  (\mm\nn)^{s-i}F^i,
\]
thanks to \eqref{eq_containmentinT}. Using \Cref{lem_decompose_ordinarypow} and the fact that $J\subseteq \nn$,
\[
x\nn^{s-i} \in (\mm\nn)^{s-i}F^i=(\mm\nn)^{s-i}(I^i+J^i+\mm\nn F^{i-1})\subseteq \mm^{s-i}I^i+\nn^{s-i+1}.
\]
The minimal generators of $x\nn^{s-i}$ have bidegree $(\deg x, s-i)$, so they all belong to $\mm^{s-i}I^i$. Hence the last chain implies 
$x\nn^{s-i} \subseteq \mm^{s-i}I^i$. In particular, $x\in \mm^{s-i}I^i:y_1^{s-i}=\mm^{s-i}I^i$. Thus we get the desired containment.

(2) Assume the contrary, that for some $1\le i\le s-1$, we have
$$
I^i\mm^{s-i-1}\nn^{s-i}\subseteq F^s.
$$
Since $I\subseteq \mm^2$, it holds that $I^i\mm^{s-i-1}\subseteq \mm^{s+i-1}\subseteq \mm^s$. Using part (1), this yields
\[
I^i\mm^{s-i-1} \subseteq (F^s:_T \nn^{s-i}) \cap (I^i\cap \mm^s)R= \mm^{s-i}I^i.
\]
But then Nakayama's lemma yields the contradiction $I^i\mm^{s-i-1}=0$. Hence the stated non-containment holds true.
\end{proof}

Now we are ready for the
\begin{proof}[Proof of \Cref{prop_depthordinarypower_0}]
We may assume that $I\neq (0)$. Applying \Cref{lem_noncontainmentsandcontainments}(2) for $i=s-1$, we have that $I^{s-1}\nn \not\subseteq F^s$. Using $I\subseteq \mm^2$ and $I+\mm\nn\subseteq F$, it is easy to see that 
$$
I^{s-1}\nn \subseteq F^s:(\mm+\nn).
$$
Thus $T/F^s$ has a non-trivial socle, and hence $\depth(T/F^s)=0$.
\end{proof}

We can provide an explicit formula for the regularity of ordinary powers of $F$.
\begin{thm}
\label{thm_reg_ordinarypow}
Let $\kk$ be a field of arbitrary characteristic. Then for  all $s\ge 1$, there is an equality
\begin{align*}
\reg F^s &= \max \limits_{i\in [1,s]} \left\{2s,\reg(\mm^{s-i}I^i)+s-i, \reg(\nn^{s-i}J^i)+s-i\right\}.
\end{align*}
If moreover either $I$ or $J$ is non-zero, then for all $s\ge 1$, there is an equality
\[
\reg F^s = \max \limits_{i\in [1,s]} \left\{\reg(\mm^{s-i}I^i)+s-i, \reg(\nn^{s-i}J^i)+s-i\right\}.
\]
\end{thm}
Note that by \cite[Theorem 5.1]{NgV19b}, \Cref{thm_reg_ordinarypow} holds true if either $\chara \kk=0$, or $I$ and $J$ are both monomial ideals. The proof employs Betti splitting arguments which are specific to the case where these extra assumptions are available. Here we give a completely different proof for the general case, without any sort of Betti splitting arguments. We will exploit special features of the zero-th local cohomology $H^0_\pp(T/F^s)$.
\begin{rem}
The first assertion in \Cref{thm_reg_ordinarypow} corrects a minor error in the statement of \cite[Theorem 5.1]{NgV19b}: To be precise, one has to assume that either $I$ or $J$ is non-zero in that result, otherwise the formula 
\[
\reg F^s = \max \limits_{i\in [1,s]} \left\{\reg(\mm^{s-i}I^i)+s-i, \reg(\nn^{s-i}J^i)+s-i\right\}.
\]
is incorrect. Indeed, when $I=(0), J=(0)$, we get $\reg F^s=2s$ for all $s$.
\end{rem}
First, we prove that $\reg F^s$ admits the upper bound given in \Cref{thm_reg_ordinarypow}.
\begin{lem}
\label{lem_upperbound}
For each $s\ge 1$, there is an inequality
\[
\reg F^s \le \max \limits_{i\in [1,s]} \left\{2s,\reg(\mm^{s-i}I^i)+s-i, \reg(\nn^{s-i}J^i)+s-i\right\}. 
\]
\end{lem}
For this, let us first recall the following result from \cite{NgV19b}.
\begin{lem}
\label{lem_exactseq}
 Denote $H=I+\mm\nn$. Let $s\ge 1$ be an integer. For each $1\le t\le s$, denote $G_t=H^s+\sum\limits_{i=1}^t (\mm\nn)^{s-i}J_i$, and $G_0=H^s$. The following statements hold.
 \begin{enumerate}[\quad \rm (1)]
  \item $F^s=H^s+\sum\limits_{i=1}^s(\mm\nn)^{s-i}J_i=G_s$.
  \item For each $1\le t\le s$, there is an equality $G_{t-1} \cap (\mm\nn)^{s-t}J^t=\mm^{s-t+1}\nn^{s-t}J^t$,  and an exact sequence
  \[
0\to G_{t-1} \to G_t \to \frac{(\mm\nn)^{s-t}J^t}{\mm^{s-t+1}\nn^{s-t}J^t} \to 0.
  \]
 \end{enumerate}
\end{lem}
\begin{proof}
Statement (1) and the first assertion of (2) follow from \cite[Proposition 4.4]{NgV19b}. Since  $G_t=G_{t-1}+(\mm\nn)^{s-t}J^t$, the exact sequence is a consequence of the first assertion of (2).
\end{proof}

\begin{proof}[Proof of \Cref{lem_upperbound}]
For each $1\le t\le s$, by \Cref{lem_exactseq}, there is an exact sequence
  \[
0\to G_{t-1} \to G_t \to \frac{(\mm\nn)^{s-t}J^t}{\mm^{s-t+1}\nn^{s-t}J^t} \cong \frac{\mm^{s-t}}{\mm^{s-t+1}}\otimes_\kk (\nn^{s-t}J^t) \to 0.
  \]
In particular, thanks to \Cref{lem_tensor}, 
\begin{align*}
\reg G_t &\le \max \left\{\reg G_{t-1}, \reg \left( \frac{\mm^{s-t}}{\mm^{s-t+1}}\otimes_\kk (\nn^{s-t}J^t)\right) \right\}\\
         &= \max \left\{\reg G_{t-1}, \reg \left( \frac{\mm^{s-t}}{\mm^{s-t+1}}\right )+\reg (\nn^{s-t}J^t) \right\}\\
         &=\max \left\{\reg G_{t-1}, \reg (\nn^{s-t}J^t)+s-t \right\}.
\end{align*}
Since $F^s=G_s, H^s=G_0$ per \Cref{lem_exactseq}, using the last chain repeatedly, we get
\begin{align*}
\reg F^s =\reg G_s &\le \max_{t\in [1,s]}\left\{\reg G_0, \reg (\nn^{s-t}J^t)+s-t \right\}\\ 
                   &= \max_{i\in [1,s]}\left\{\reg H^s, \reg (\nn^{s-i}J^i)+s-i \right\}.
\end{align*}
Now $H=I+\mm\nn$ can be seen as the fiber product of $I\subseteq R$ and $(0)\subseteq S$. Hence by similar arguments, we get
\begin{align*}
\reg H^s & \le \max_{i\in [1,s]}\left\{\reg (\mm\nn)^s, \reg (\mm^{s-i}I^i)+s-i \right\} =\max_{i\in [1,s]}\left\{2s, \reg (\mm^{s-i}I^i)+s-i \right\}.
\end{align*}
Combining the last two displays, we obtain the desired upper bound for $\reg F^s$.
\end{proof}

Next, we prove the (harder) reverse inequality in \Cref{thm_reg_ordinarypow}, namely the lower bound for $\reg F^s$. For each $s\ge 1$, denote $U_s=(I+\nn)^s\cap (J+\mm)^s$. Note that $U_1=I+J+\mm \nn=F$. The ideal $U_s$ is crucial to the analysis of the lower bound for $\reg F^s$, which is done by the Lemmas \ref{lem_decomposition_G} -- \ref{lem_lowerbound_UsmodeFs}.

\begin{lem}
\label{lem_decomposition_G}
 For each $s\ge 1$, there is an equality
\[
U_s= \sum_{i=0}^s\sum_{t=0}^s (I^i\cap \mm^{s-t})(J^t\cap \nn^{s-i}).
\]
\end{lem}
\begin{proof}
Apply \Cref{lem_intersect_formula} for the filtration of ordinary powers $I^\bullet$ and $J^\bullet$. 
\end{proof}

\begin{lem}
\label{lem_regUs}
For each $s\ge 1$, there is an equality
\[
\reg U_s = \max \limits_{i\in [1,s]} \left\{2s,\reg(I^i)+s-i, \reg(J^i)+s-i\right\}.
\]
\end{lem}
\begin{proof}
We consider the following exact sequence
\[
0\to \frac{T}{U_s} \to \frac{T}{(I+\nn)^s}\bigoplus  \frac{T}{(J+\mm)^s} \to \frac{T}{(I+\nn)^s+(J+\mm)^s}\to 0.
\]
The argument for the desired equality  is similar to the proof of \Cref{thm_depth_reg}(ii), taking \Cref{lem_depthregIplusn_ordinary} into account. We give only a sketch here.

If $I=J=(0)$, then $U_s=\mm^s\nn^s$ has regularity $2s$, as expected. Assume, without loss of generality, that $I\neq (0)$. Then we can show that the artinian module $T/\left((I+\nn)^s+(J+\mm)^s\right)$ has regularity at most $2s-2$. On the other hand, per \Cref{lem_depthregIplusn_ordinary},
\[
\reg \frac{T}{(I+\nn)^s} \ge \reg \frac{R}{I^s}=\reg I^s-1 \ge 2s-1.
\]
Hence the last exact sequence yields
\[
\reg \frac{T}{U_s} =\max \left\{\reg \frac{T}{(I+\nn)^s},\reg \frac{T}{(J+\mm)^s}\right\} \ge 2s-1.
\]
Invoking \Cref{lem_depthregIplusn_ordinary}, we are done.
\end{proof}

\begin{lem}
\label{lem_finitelength}
For each $s\ge 1$, the following statements hold.
\begin{enumerate}[\quad \rm (1)]
 \item There are containments 
$$
\pp^{2s-1}U_s \subseteq F^s\subseteq U_s.
$$
In particular, $U_s/F^s$ has finite length over $T$.
\item There is an equality
\[
\reg F^s= \max_{i\in [1,s]} \left\{\reg(U_s/F^s)+1, 2s, \reg(I^i)+s-i, \reg(J^i)+s-i\right\}.
\]
\end{enumerate}
\end{lem}
\begin{proof}
(1): It suffices to establish the first assertion, in the case $s\ge 2$. Since $F=(I+\nn)\cap (J+\mm)$, $F^s\subseteq (I+\nn)^s\cap (J+\mm)^s=U_s$. We are left with the containment
\[
 \pp^{2s-1}U_s \subseteq F^s.
\]
By \Cref{lem_decomposition_G}, we have to show that for each $0\le i,t\le s$ and each $0\le j\le 2s-1$, 
\[
\mm^j\nn^{2s-j-1}(I^i\cap \mm^{s-t})(J^t\cap \nn^{s-i})\subseteq F^s.
\]
Either $j\ge s$ or $2s-j-1\ge s$. In the first case,
\[
\mm^j\nn^{2s-j-1}(I^i\cap \mm^{s-t})(J^t\cap \nn^{s-i}) \subseteq \mm^s I^i\nn^{s-i} \subseteq I^i(\mm\nn)^{s-i} \subseteq F^s.
\]
Arguing similarly for the remaining case, we conclude the proof of (1).

(2): From the exact sequence
\[
0\to \frac{U_s}{F^s} \to \frac{T}{F^s} \to \frac{T}{U_s} \to 0
\]
and the fact that $U_s/F^s$ has finite length, we get (see \cite[Corollary 20.19]{Eis95}) that
\[
\reg(T/F^s)=\max \left\{\reg(U_s/F^s),\reg(T/U_s) \right\}.
\]
Therefore $\reg F^s=\max \left\{\reg(U_s/F^s)+1,\reg(U_s) \right\}$, and we are done by invoking \Cref{lem_regUs}.
\end{proof}

\begin{lem}
\label{lem_regUsmodeFs_ge2sminus1}
Assume that $I\neq (0)$. Then for every $s\ge 2$, there is an inequality 
\[
\reg \frac{U_s}{F^s} \ge 2s-1.
\]
\end{lem}
\begin{proof}
Since $I\neq (0)$, applying \Cref{lem_noncontainmentsandcontainments}(2) for $i=s-1$, we get $I^{s-1}\nn \not\subseteq F^s$. The exact sequence of artinian modules
\[
0\to \frac{I^{s-1}\nn +F^s}{F^s} \to \frac{U_s}{F^s}
\]
yields the chain
\[
\reg  \frac{U_s}{F^s} \ge \reg \frac{I^{s-1}\nn +F^s}{F^s} \ge 2s-1.
\]
The last inequality follows from the fact that $I^{s-1}\nn \subseteq \mm^{2s-2}\nn$, which is generated in degree $2s-1$.
\end{proof}

The key step in the proof of \Cref{thm_reg_ordinarypow} is accomplished by

\begin{lem}
\label{lem_lowerbound_UsmodeFs}
Assume that $I \neq (0)$. Then for every integer $s\ge 2$, there is an inequality
\[
\reg \frac{U_s}{F^s} \ge \max_{i\in [1,s]} \left\{2s-1, \reg \frac{I^i}{\mm^{s-i}I^i}+s-i \right\}.
\]
\end{lem}
\begin{proof}
That  $\reg(U_s/F^s)\ge 2s-1$ follows from \Cref{lem_regUsmodeFs_ge2sminus1} thanks to the hypothesis $I\neq (0)$. It remains to show that for every $1\le i\le s$,
\[
\reg \frac{I^i}{\mm^{s-i}I^i}+s-i \le \reg \frac{U_s}{F^s}.
\]
We may assume that $i\le s-1$, since otherwise the left-hand side is $-\infty$.

Since $\mm^{s-i}I^i\subseteq I^i\cap \mm^s$, there is an exact sequence of artinian $R$-modules
\[
0\to \frac{I^i\cap \mm^s}{\mm^{s-i}I^i} \to \frac{I^i}{\mm^{s-i}I^i} \to \frac{I^i+\mm^s}{\mm^s} \to 0.
\]
Hence denoting $W_s=\dfrac{I^i\cap \mm^s}{\mm^{s-i}I^i}$,
\begin{align*}
s-i+\reg \frac{I^i}{\mm^{s-i}I^i} &=\max \left\{\reg W_s+s-i, \reg \frac{I^i+\mm^s}{\mm^s}+s-i \right\}\\
                                  &\le \max \{\reg W_s+s-i, 2s-i-1\}. 
\end{align*}
Since $2s-i-1 < 2s-1\le \reg (U_s/F^s)$, it remains to show that
\[
\reg W_s+ s-i \le \reg \frac{U_s}{F^s}.
\]
It is harmless to assume that $W_s\neq 0$. Denote $d=\reg W_s$. Since $W_s$ has finite length, there exists a homogeneous element $x\in (I^i\cap \mm^s) \setminus (\mm^{s-i}I^i)$ that is of degree $d$. Per \Cref{lem_noncontainmentsandcontainments}(2),
\[
(F^s:_T \nn^{s-i})\cap (I^i\cap \mm^s)R=\mm^{s-i}I^i.
\]
Therefore $x\notin  F^s:_T \nn^{s-i}$, namely $x\nn^{s-i}\not\subseteq F^s$. We note that
\[
x\nn^{s-i} \subseteq (I^i\cap \mm^s)\nn^{s-i} \subseteq U_s,
\]
thanks to \Cref{lem_decomposition_G}. Thus $U_s/F^s$ contains a non-zero homogeneous element of degree $d(x\nn^{s-i})=d+s-i=\reg W_s+s-i$. This implies
\[
\reg \frac{U_s}{F^s} \ge \reg W_s+ s-i 
\]
and concludes the proof.
\end{proof}
Now we are ready for the 
\begin{proof}[Proof of \Cref{thm_reg_ordinarypow}]
For the first assertion, if $s=1$, the equation $\reg F=\max\{2, \reg I, \reg J\}$ follows from \Cref{lem_decomposition_F}.

Now assume that $s\ge 2$. Note that the upper bound
\[
\reg F^s \le \max \limits_{i\in [1,s]} \left\{2s,\reg(\mm^{s-i}I^i)+s-i, \reg(\nn^{s-i}J^i)+s-i\right\}. 
\]
was proved in \Cref{lem_upperbound}. If $I=J=(0)$, then $F=\mm\nn$, and the desired conclusion $\reg F^s=2s$ is clearly true. So without loss of generality, we may assume that $I\neq (0)$.

Applying \Cref{lem_finitelength}(2), we get
\[
\reg F^s= \max_{i\in [1,s]} \left\{\reg(U_s/F^s)+1, 2s, \reg(I^i)+s-i, \reg(J^i)+s-i\right\}.
\]
Together with \Cref{lem_lowerbound_UsmodeFs}, which is applicable since $I\neq (0)$, we get the inequality below
\begin{align*}
\reg F^s &\ge \max_{i\in [1,s]} \left\{2s, \reg \frac{I^i}{\mm^{s-i}I^i}+s-i+1, \reg(I^i)+s-i, \reg(J^i)+s-i\right\}\\
         &= \max_{i\in [1,s]} \left\{2s, \reg(\mm^{s-i}I^i)+s-i, \reg(J^i)+s-i\right\}.
\end{align*}
The equality follows from \Cref{lem_regmM}. If $J=(0)$ then we get the desired lower bound for $\reg F^s$. If $J\neq (0)$, then arguing by symmetry, we also get the desired lower bound. Hence the first assertion holds true.

The second assertion follows since $\reg I^s\ge d(I^s)\ge 2s$ if $I\neq (0)$. The proof of the theorem is concluded. 
\end{proof}
\begin{rem}
\label{rem_noTor-vanishing}
The proof of \cite[Theorem 5.1]{NgV19b} is based crucially on the fact that if either $\chara \kk=0$ or $I$ is a monomial ideal, then for all $1\le t\le s$, the map $\mm^{s-t}I^t \to \mm^{s-t+1}I^{t-1}$ is \emph{Tor-vanishing}, namely the induced map on Tor against $\kk$ is zero. In \cite[Question 4.2]{NgV19b}, it was asked whether this is also true in positive characteristic. However, the answer is ``No!'', thanks to \cite[Example 3.9]{NgHo22}. In characteristic 2, the map $I^2\to I$ need not be Tor-vanishing, hence so neither is the map $I^2\to \mm I$ (in general, clearly if a map of graded $R$-modules $M\to P$ factors through a Tor-vanishing map $M\to N$, then it is Tor-vanishing itself).

The proof of \Cref{thm_reg_ordinarypow}, however, does not depend on Tor-vanishing arguments.
\end{rem}

As a corollary, we can now present the proof of the formulas for depth and regularity of symbolic powers of $F$ in the case $\min\{\depth(R/I),\depth(S/J)\}=0$.
\begin{proof}[Proof of \Cref{thm_depth_zerodepthcase}]
The equality $F^{(s)}=F^s$ for all $s\ge 1$ follows from part (2) of \Cref{lem_decomposition}. Since $R$ and $S$ have positive Krull dimensions, and  $\min\{\depth(R/I),$ $\depth(S/J)\}=0$, either $I$ or $J$ is non-zero. The remaining equalities follow from \Cref{prop_depthordinarypower_0} and \Cref{thm_reg_ordinarypow}.
\end{proof}
Thanks to \Cref{thm_reg_ordinarypow}, various results in \cite{NgV19b} become valid regardless of the characteristic of $\kk$. For example, we have the following improvement of \cite[Corollaries 5.2 and 5.6]{NgV19b}, with nearly identical proofs.
\begin{cor}
\label{cor_reg_equigen}
Keep using Notation \ref{notn_IandJ}. Assume further that each of $I$ and $J$ is non-zero and generated by forms of the same degree. Then for all $s\ge 1$, there is an equality
\[
\reg F^s=\max_{i\in [1,s]}\bigl\{\reg I^i+s-i, \reg J^i+s-i\bigr\}.
\]
\end{cor}

\begin{cor}
\label{cor_asymptotic_regordpower}
Keep using \Cref{notn_IandJ}. Assume further that both $I$ and $J$ are non-zero ideals satisfying one of the following conditions: 
\begin{enumerate}[\quad \rm (i)]
\item All the minimal homogeneous generators have degree $2$;
\item All the minimal homogeneous generators have degree at least $3$;
\item The subideal generated by elements of degree 2 is integrally closed, e.g. $I$ and $J$ are squarefree monomial ideals.
\end{enumerate}
Then for all $s\gg 0$, there is an equality $\reg F^s=\max\{\reg I^s,\reg J^s\}$.
\end{cor}
The details are left to the interested reader.


\section{Remarks and questions}
\label{sect_remarks}
While we focus on symbolic powers defined using associated primes, our method can be adapted to the symbolic powers defined using minimal primes
\[
\up{m}I^{(s)}=\bigcap\limits_{P \in \Min_R(I)}(I^sR_P\cap R).
\]
Denote $I^{\um}=\up{m}I^{(1)}$ the \emph{unmixed part} of the ideal $I$. Thus $I=I^{\um}$ if and only if $I$ is unmixed. Modifying \Cref{lem_decomposition} and \Cref{lem_decomposition_as_sum}, we have
\begin{lem} 
\label{lem_decomposition_minsymbpow}
Employ \Cref{notn_IandJ}. Assume that $\min\left\{\dim(R/I), \dim(S/J)\right\}\ge 1$, i.e. $I$ is not $\mm$-primary and $J$ is not $\nn$-primary. Then there are equalities
\[
\up{m}F^{(1)}=(\up{m}I^{(1)}+\nn)\cap (\up{m}J^{(1)}+\mm)=\up{m}I^{(1)}+\up{m}J^{(1)}+\mm\nn,
\]
and for every integer $s\ge 1$, we have equalities
$$
\up{m}F^{(s)} = \, \up{m}(I+\nn)^{(s)} \cap \, \up{m}(J+\mm)^{(s)}=(I^{\um}+\nn)^{(s)}\cap (J^{\um}+\mm)^{(s)}.
$$
\end{lem}
\begin{lem}
\label{lem_decomposition_as_sum_minsymbpow}
Employ \Cref{notn_IandJ}. Assume that $\min\left\{\dim(R/I),\dim(S/J)\right \}\ge 1$. For all $s\ge 1$, there is an equality
\[
\up{m}F^{(s)}= \sum_{i=0}^{s}\sum_{t=0}^{s}(\up{m}I^{(i)}\cap \mm^{s-t})(\up{m}J^{(t)}\cap \nn^{s-i}).
\]
\end{lem}
The suitable modification of \Cref{thm_depth_reg} for ``minimal'' symbolic powers is
\begin{thm} 
\label{thm_depth_reg_minsymbpow}
Employ \Cref{notn_IandJ}. Assume that $\min\left\{\dim(R/I), \dim(S/J)\right\}\ge 1$. Then for every integer $s\geq 1$, there are equalities
\begin{itemize}
\item[(i)] ${\rm depth}(\up{m}F^{(s)})=2$, and
\item[(ii)] ${\rm reg}(\up{m}F^{(s)})=\max\limits_{i\in [1,s]}\left\{2s,\reg(\up{m}I^{(i)})+s-i, {\rm reg}(\up{m}J^{(i)})+s-i \right\}$.
\end{itemize}
\end{thm}
We leave the details of the proofs to the interested reader.

The following question came up in the course of proving Corollary \ref{cor_reg_Fs_unmixedcase}. We do not have an answer to it yet.
\begin{quest}
Let $I\subseteq \mm^2$ be an unmixed homogeneous ideal in a polynomial ring $R$. Is it true that for all $s\ge 1$, the inequality
$\reg I^{(s)}\ge 2s$ holds?
\end{quest}

We may ask whether the complicated formula for regularity in \Cref{thm_depth_reg}
$$
\reg(F^{(s)})=\max\limits_{i\in [1,s]} \left\{2s,\reg(I^{(i)})+s-i, \reg(J^{(i)})+s-i\right \}
$$
can be simplified to 
$$
\reg F^{(s)}=\max\{\reg I^{(s)}, \reg J^{(s)}\}
$$ 
for all $s\ge 1$, at least when $I\subseteq \mm^2$ and $J\subseteq \nn^2$. Unfortunately, this is not true even if both $I$ and $J$ are primary binomial ideals.
\begin{ex}
Let $R=\kk[a,b,c,d], S=\kk[y,z]$, 
\begin{align*}
I &= (a^5,a^4b,ab^4,b^5, a^2b^3c^7-a^3b^2d^7,a^3b^3),J= (y^2).
\end{align*}
Hence $T=\kk[a,b,c,d,y,z]$, $F=I+J+(a,b,c,d)(y,z)$. We can check that $I$ and $J$ are primary ideals, $\sqrt{I}=(a,b), \sqrt{J}=(y)$, 
\begin{align*}
\depth(R/I)&=2, \depth(S/J)=1,\\
\reg I &= 12, \reg I^{(2)}=10, \reg J=2, \reg J^{(2)}=4,\\
\reg F &= 12, \reg F^{(2)}=13.
\end{align*}
Hence $\reg F^{(2)}=13 > \max\{\reg I^{(2)},\reg J^{(2)}\}=\max\{10,4\} =10.$ 
\end{ex}
We also have a similar example where each of $I$ and $J$ is a monomial ideal generated in a single degree.
\begin{ex}
Let $R=\kk[a,b,c,d,e,f], S=\kk[y,z]$, 
\begin{align*}
I &= (a^4,a^3b,ab^3,b^4)(c,d,e)^7+a^2b^2(c^7,d^7,e^7),J= (y^2).
\end{align*}
Hence $T=\kk[a,b,c,d,e,f,y,z]$, $F=I+J+(a,b,c,d,e,f)(y,z)$. We can check that $I$ and $J$ are equigenerated monomial ideals, 
\begin{align*}
\depth(R/I)&=1, \depth(S/J)=1,\\
\reg I &= 23, \reg I^{(2)}=22, \reg J=2, \reg J^{(2)}=4,\\
\reg F &= 23, \reg F^{(2)}=24.
\end{align*}
That $\reg I^{(2)}=22$ can be seen using $\Ass_R(I)=\{(a,b), (c,d,e), (a,b,c,d,e)\}$, $I^2=(a,b)^8(c,d,e)^{14}$, and thus $I^{(2)}=I^2$ thanks to \cite[Lemma 2.2]{HJKN}. Hence $\reg F^{(2)}=24 > \max\{\reg I^{(2)},\reg J^{(2)}\}=\max\{22,4\} =22.$ 
\end{ex}
In view of \Cref{cor_asymptotic_regordpower}, we may ask
\begin{quest}
\label{quest_reg_larsymbpow}
Let $I\subseteq \mm^2$ and $J\subseteq \nn^2$ be homogeneous ideals of $R$ and $S$, respectively. Assume furthermore that both $I$ and $J$ are unmixed. Is it true that for all $s\gg 0$, the equality
\[
\reg F^{(s)}=\max\{\reg I^{(s)}, \reg J^{(s)}\}
\]
holds?
\end{quest}
Note that from \cite[Remark 5.7]{NgV19b}, it may happens for mixed monomial ideals $I$ and $J$ that $\reg F^s > \max\{\reg I^s, \reg J^s\}$ for all $s\ge 3$. This is the reason why we require that $I$ and $J$ are unmixed in the last question. Nevertheless, we do not know of any counterexample to this question even among mixed ideals.
\begin{rem}
\Cref{quest_reg_larsymbpow} would have a positive answer if we can show that for an unmixed homogeneous ideal $I\subseteq \mm^2$, it holds that
\[
\reg I^{(s)}=\max\limits_{i\in [1,s]}\{\reg I^{(i)}+s-i\}
\]
for all $s\gg 0$. We do not whether the last statement is always true, even if $I$ is an unmixed monomial ideal.
\end{rem}
There are exact formulas for the depth and regularity of ordinary and symbolic powers of $I+J$ in \cite{HNTT, HJKN, NgV19a}. These results, however, depend on Tor-vanishing results, and require that either $\chara \kk=0$, or $I$ and $J$ are both monomial ideals. In view of the main results of this paper, it would be interesting to see whether we can prove such formulas in a characteristic-independent way.


\section*{Acknowledgments}
The first and the second named authors (HVD and HDN) were supported by NAFOSTED under the grant number 101.04-2023.30. The third named author is supported by the FAPA grant from the Universidad de los Andes. Part of this work was done when the third named author visited Vietnam Academy of Science and Technology. He gratefully thanks their hospitality.



\end{document}